\documentclass[12pt]{article}
\usepackage{e-jc}

\usepackage[colorlinks=true,citecolor=black,linkcolor=black,urlcolor=blue]{hyperref}

\dateline{June 15, 2019}{TBD}{TBD}


\Copyright{The author. Released under the CC BY license (International 4.0).}

\title{Voting Rules That Are Unbiased But Not Transitive-Symmetric}

\author{Aadyot Bhatnagar\\
	\small California Institute of Technology\\[-0.8ex] 
	\small Pasadena, CA, U.S.A.\\
	\small\tt aadyotb@gmail.com}

\usepackage{amssymb,amsmath,physics,dsfont,amsthm}
\usepackage{graphicx,subfig,ctable,booktabs,float}
\usepackage{xifthen}

\usepackage{color}	
\definecolor{deepblue}{rgb}{0,0,0.5}
\definecolor{deepred}{rgb}{0.6,0,0}
\definecolor{deepgreen}{rgb}{0,0.5,0}

\newcommand{\Z}{\mathbb{Z}}

\newcommand{\R}{\mathbb{R}}

\newcommand{\I}{\mathbb{I}}

\newcommand{\one}{\mathds{1}}

\newcommand{\E}[2][]{
	\ifthenelse {\isempty{#1}{}}
	{\mathbb{E}}
	{\operatorname*{\mathbb{E}}_{#1}}
	\left[{#2}\right]
}

\renewcommand{\P}[2][]{
	\ifthenelse {\isempty{#1}{}}
	{\mathbb{P}}
	{\operatorname*{\mathbb{P}}_{#1}}
	\left[{#2}\right]
}

\begin{document}
	
	\maketitle
	
	\begin{abstract}
		We explore the relation between two natural symmetry properties of voting rules. The first is transitive-symmetry --- the property of invariance to a transitive permutation group --- while the second is the ``unbiased'' property of every voter having the same influence for all i.i.d.\ probability measures. We show that these properties are distinct by two constructions --- one probabilistic, one explicit --- of rules that are unbiased but not transitive-symmetric.
	\end{abstract}
	
	\section{Introduction}
	
	We study {\em voting rules}: functions $\Phi : \{-1,1\}^n \to \{-1,1\}$ that map a voting profile --- the preferences of a set $V = [n]$ of $n$ voters between alternatives $-1$ and $1$ --- to an outcome in $\{-1, 1\}$. We restrict our attention to voting rules $\Phi$ that are {\em odd}, i.e.\ such that $\Phi(-x) = - \Phi(x)$, and {\em monotone}, i.e.\ such that $\Phi(x) \ge \Phi(x')$ whenever $x \ge x'$ coordinate-wise.
	
	Let $S_n$ be the set of permutations over $V$. For a permutation $\sigma \in S_n$, we define a permutation of a voting rule $\Phi$ as $(\sigma \Phi)(x) = \Phi(\sigma^{-1}(x))$. We say that $\sigma$ is an {\em automorphism} of $\Phi$ if $\sigma \Phi = \Phi$. We denote by $\mathrm{aut}(\Phi)$ the {\em automorphism group} of $\Phi$. A voting rule $\Phi$ is said to be {\em transitive-symmetric} if $\mathrm{aut}(\Phi)$ is transitive, i.e.\ for every pair of voters $u, v \in V$, there is an automorphism $\sigma \in \mathrm{aut}(\Phi)$ such that $\sigma(u) = v$. Intuitively, transitive-symmetric voting rules define ``equitable'' elections where each voter plays the same ``role'' in determining the final outcome. These objects were studied in detail by \cite{tamuz2019equitable}, who call them ``equitable voting rules.''
	
	For any voting rule $\Phi$ and a probability measure $q$ over $\{-1,1\}^n$, voter $i$'s {\em influence} on $\Phi$ is the probability that they are {\em pivotal} for a voting profile chosen from $q$. More formally, for any $x \in \{-1, 1\}^n$, let $x^{\oplus i} = (x_1, \ldots, x_{i-1}, -x_i, x_{i+1}, \ldots, x_n)$ be the voting profile obtained when only voter $i$ changes their vote from $x$. We say $i$ is {\em pivotal} for $x$ in $\Phi$ if $\Phi(x) \ne \Phi(x^{\oplus i})$, and the {\em influence} of $i$ on $\Phi$ with respect to probability measure $q$ on $\{-1, 1\}^n$ is
	\[ \I_i^p(\Phi) = \P[x \sim q]{\Phi(x) \ne \Phi(x^{\oplus i})} \]
	We say that a voting rule $\Phi$ is {\em unbiased} if $\I_1^q(\Phi) = \cdots = \I_n^q(\Phi)$ for every i.i.d.\ probability measure $q = p^{\otimes n}$ on $\{-1, 1\}^n$. It is easy to see that all transitive-symmetric voting rules are unbiased, but it is unclear whether the converse holds. Our main theorem shows that the two properties are distinct:\footnote{We only consider odd $n$ because it is a long-standing open question for which even $n$ there exist odd, monotone, transitive-symmetric voting rules \cite{isbell1960homogeneous}.}
	\begin{theorem}
		\label{thm:main}
		For all $n$ odd and large enough, there exist odd, monotone voting rules $\Phi : \{-1, 1\}^n \to \{-1, 1\}$ that are unbiased but not transitive-symmetric.
	\end{theorem}
	
	We prove Theorem \ref{thm:main} by introducing a class of {\em graphic} voting rules defined by an underlying graph (Section \ref{sec:prelim}). Informally, we associate transitive-symmetric voting rules to vertex-transitive graphs, and unbiased voting rules to regular graphs. This construction reduces our task to proving the existence of $d$-regular graphs that satisfy certain properties. In Section \ref{sec:prob}, we provide a probabilistic existence proof inspired by \cite{kim2002asymmetry} showing that the graphic voting rules of random regular graphs are unbiased and have trivial automorphism groups with probability $1 - o(1)$. Section \ref{sec:explicit} provides an explicit construction of a voting rule that is unbiased but not transitive-symmetric by composing a specific asymmetric graphic voting rule with a transitive-symmetric voting rule. However, this construction is only valid for specific $n$. Section \ref{sec:conclusion} concludes.
	
	\section{Preliminaries}
	\label{sec:prelim}
	
	\subsection{Graph Theory}
	Let $G = (V, E)$ be an unweighted, undirected graph with vertex set $V = [n]$ and edge set $E$. We intentionally overload notation because we will associate each vertex of a graph $G$ to a voter in a voting rule $\Phi_G$. For any two vertices $u, v \in V$, let $d(u, v)$ denote the length of a shortest path between $u$ and $v$ in $G$. The ball of radius $r$ about a vertex $v$ in $G$ is the set of all points a distance at most $r$ away from $v$, i.e.\ $B_r(v) = \{ u \in V : d(u, v) \le r \}$. Analogously, we also let $B_r(G) = \{ B_r(v) : v \in V \}$ denote the set of all balls of radius $r$ in $G$. The {\em diameter} of $G$ is $\mathrm{diam}(G) = \max_{u, v \in V} d(u, v)$, and the {\em radius} of $G$ is $\mathrm{rad}(G) = \lceil \mathrm{diam}(G) / 2 \rceil$. While this definition of the graph radius is non-standard, it has the attractive property that $\mathrm{rad}(G) = \min \{ r \in \R : B_r(u) \cap B_r(v) \ne \varnothing \,\, \forall u, v \in V \}$. 
	
	For a permutation $\sigma \in S_n$, let $\sigma(G) = (V, \{ \{\sigma(u), \sigma(v)\} : \{u, v\} \in E \})$ be the graph obtained by using $\sigma$ to permute the endpoints of every edge in $G$. If $\sigma(G) = G$, then we say that $\sigma$ is an {\em automorphism} of $G$, or $\sigma \in \mathrm{aut}(G)$. Analogously, if $B_r(\sigma(G)) = B_r(G)$, then we say that $\sigma$ is an {\em automorphism} of $B_r(G)$, or $\sigma \in \mathrm{aut}(B_r(G))$. Note that while $\mathrm{aut}(G) \subseteq \mathrm{aut}(B_r(G))$, it is unclear when the converse holds. Figure \ref{fig:regular} provides two examples where $\mathrm{aut}(G) = \mathrm{aut}(B_1(G))$.
	
	We say that a graph is {\em vertex-transitive} if for every pair of vertices $u, v \in V$, there is an automorphism $\sigma \in \mathrm{aut}(G)$ such that $\sigma(u) = v$. We say that $G$ is {\em asymmetric} if $\abs{\mathrm{aut}(G)} = 1$, i.e.\ the identity is the only automorphism of $G$. We define vertex-transitivity and asymmetry for $B_r(G)$ analogously. Finally, we call a graph $d$-regular if all of its vertices have degree $d$. It is easy to see that all vertex-transitive graphs must be regular, though the converse need not hold (Figure \ref{fig:regular}).
	
	\begin{figure}
		\centering
		\subfloat[This 3-regular graph is not vertex-transitive, though it still has nontrivial automorphisms. 
		\label{fig:reg_not_vtxtrans}]
		{
			\centering
			\includegraphics[width=0.45\textwidth]{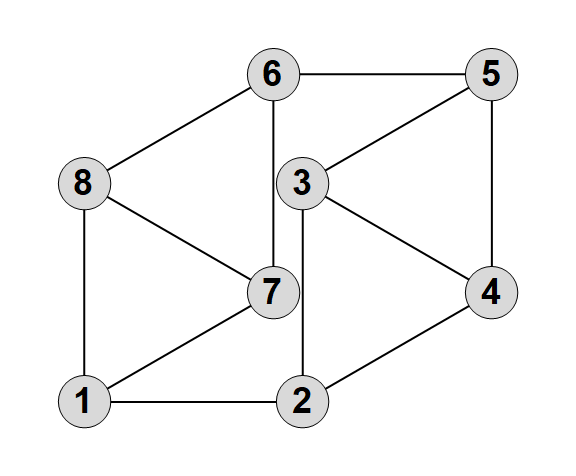}
		} \qquad
		\subfloat[This 4-regular graph is totally asymmetric, i.e.\ it has no non-trivial automorphisms.
		\label{fig:reg_asymmetric}]
		{
			\centering
			\ \ \includegraphics[width=0.42\textwidth]{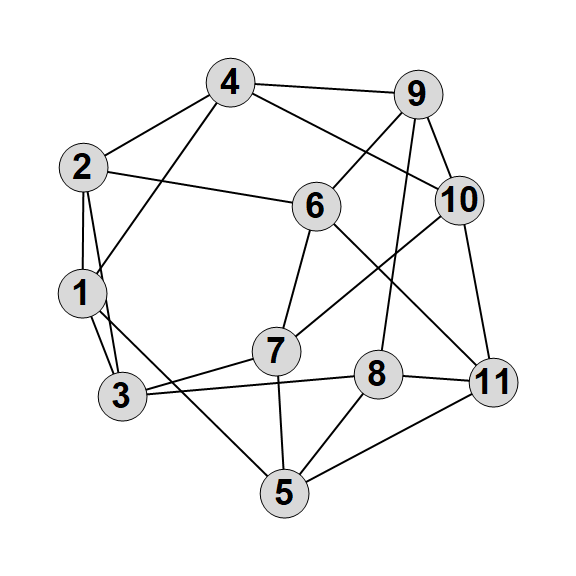} \ \ 
		}
		\caption{Regular graphs that are not vertex-transitive and have $\mathrm{aut}(G) = \mathrm{aut}(B_1(G))$. The graphic voting rule associated with the graph in Figure \ref{fig:reg_asymmetric} is unbiased but not transitive-symmetric.}
		\label{fig:regular}
	\end{figure}
	
	\subsection{Winning Coalitions of Voting Rules}
	A {\em winning coalition} for a voting rule $\Phi$ is a set of voters $S \subseteq [n]$ such that if a voting profile $x \in \{-1, 1\}^n$ has $x_i = y$ for all $i \in S$, then $\Phi(x) = y$. We say that a winning coalition $S$ is {\em minimal} if no strict subset $T \subset S$ of it is also a winning coalition. We note that $\Phi$ is odd and monotone if and only if the winning coalitions for both alternatives $-1$ and $1$ are equal, and any superset of a winning coalition is also a winning coalition. This observation gives rise to the following characterization of odd, monotone voting rules:
	\begin{definition}
		For any odd $n$, let $\mathcal{F} \subseteq 2^{[n]}$ be a family of sets such that $T \cap S \ne \varnothing$ and $S \not \subset T$ for all $S, T \in \mathcal{F}$. Define the voting rule $\Phi_{\mathcal{F}} : \{-1, 1\}^n \to \{-1, 1\}$ as
		\[ \Phi_{\mathcal{F}}(x) = 
		\begin{cases}
		y & \exists S \in \mathcal{F} \,\, \mathrm{s.t.} \,\, x_i = y \,\, \forall i \in S \\ 
		\mathrm{maj}_n(x) & \mathrm{otherwise}
		\end{cases} \]
	\end{definition}
	$\Phi_{\mathcal{F}}$ is well-defined only if $\mathcal{F}$ is an intersecting family of sets. This is because every $S \in \mathcal{F}$ is a minimal winning coalition for $\Phi_{\mathcal{F}}$ (minimality follows because $S \not \subset T$ for any $S, T \in \mathcal{F}$). We call a voter $i$ {\em pivotal} for a winning coalition $S$ if $S \setminus \{i\}$ is not a winning coalition. This re-framing of voting rules in terms of winning coalitions gives us the following characterization of odd, monotone, unbiased voting rules:
	
	\begin{lemma}
		\label{lem:unbiased}
		Let $\Phi : \{-1, 1\}^n \to \{-1, 1\}$ be an odd, monotone voting rule with minimal winning coalitions $\mathcal{F}$. $\Phi$ is unbiased if and only if every voter $i$ is pivotal for an equal number of winning coalitions of each size. 
	\end{lemma}
	
	\begin{proof}
		Let $\mathcal{F}_i = \{ T : \exists ! S \in \mathcal{F} \mbox{ s.t. } i \in S \mbox{ and } S \subseteq T \}$ be the set of winning coalitions of $\Phi$ for which voter $i$ is pivotal, and let $\mathcal{F}_{i,j} = \{ S \in \mathcal{F}_i : \abs{S} = j \}$ be the subset of these coalitions with size $j$. If we let $m_{i,j} = \abs{\mathcal{F}_{i,j}}$, then for any probability measure $p$ on $\{-1, 1\}$, voter $i$ has influence
		\begin{align*}
		\I^{p^{\otimes n}}_i(\Phi) &= \sum_{S \in \mathcal{F}_i} \P[x \sim q^{\otimes n}]{ (x_j = y \,\, \forall j \in S \setminus \{i\}) \land (x_j = -y \,\, \forall j \notin S \setminus \{i\}) } \\
		&\quad + \sum_{S \in \mathcal{F}_i} \P[x \sim q^{\otimes n}]{(x_j = y \,\, \forall j \in S) \land (x_j = -y \,\, \forall j \notin S)} \\	&= \sum_{j = 1}^{n} m_{i,j} c_{j}(p)
		\end{align*}
		where, with the slight abuse of notation  $\P[x \sim p]{x = 1} = p$, the coefficients are $c_j(p) = p^{j-1} (1 - p)^{n-j+1} + p^{n-j+1} (1 - p)^{j-1} + p^j (1 - p)^{n-j} + p^{n-j} (1 - p)^j$. Because $\Phi$ is unbiased, we know that $\sum_{j = 1}^{n} m_{1,j} c_j(p) = \cdots = \sum_{j = 1}^{n} m_{n,j} c_j(p)$ for all $p$. This is possible if and only if $m_{1,j} = \cdots = m_{n,j}$ for each $j \in [n]$.
		
	\end{proof}
	
	We can also lower bound the size of the smallest winning coalition of any odd, monotone, unbiased voting rule:
	\begin{lemma}
		\label{lem:lowerbound}
		Let $\Phi : \{-1, 1\}^n \to \{-1, 1\}$ be an odd, monotone, unbiased voting rule with minimal winning coalitions $\mathcal{F}$. Then, $\abs{S} \ge \lceil \sqrt{n} \rceil$ for all $S \in \mathcal{F}$. 
		
	\end{lemma}
	
	\begin{proof}
		Let $k = \min\{ \abs{S} : S \in \mathcal{F} \}$ be the size of a smallest winning coalition of $\Phi$ (call it a {\em minimum} winning coalition) and let $\mathcal{F}^* = \bigcup_{i = 1}^{n} \mathcal{F}^*_i$ be the set of all minimum winning coalitions. Let $m_i = \abs{ \{ S \in \mathcal{F}^* : i \in S \} }$ and $m = \abs{\mathcal{F}^*}$. Since $\Phi$ is unbiased and every participating voter is pivotal for a minimum winning coalition, $m_1 = \cdots = m_n = m k / n$ by Lemma \ref{lem:unbiased} and the fact that $\sum_{i = 1}^{n} m_i = m k$. 
		
		Therefore, for any $S \in \mathcal{F}^*$, $\sum_{i \in S} \sum_{T \in \mathcal{F}} \one[i \in T] = \sum_{i \in S} m_i = m k^2 / n$. This implies that the elements of $S$ collectively appear in at most $m k^2 / n$ different $T \in \mathcal{F}^*$. Finally, because $\mathcal{F}^* \subseteq \mathcal{F}$ is an intersecting family of sets with $\abs{\mathcal{F}^*} = m$, we can conclude that $m \le mk^2 / n \implies k \ge \sqrt{n}$. Since $k \in \Z$, this implies that $k \ge \lceil \sqrt{n} \rceil$.
		
	\end{proof}
	
	It is worth noting that \cite{tamuz2019equitable} proved that $\abs{S} \ge \lceil \sqrt{n} \rceil$ for all $S \in \mathcal{F}$ when $\Phi$ is transitive-symmetric (Theorem 2), and moreover, there exist transitive-symmetric $\Phi$ for which this lower bound is tight (Theorem 6). Our Lemma \ref{lem:lowerbound} has just shown that the lower bound even holds under the weaker condition $\Phi$ is unbiased. Our Theorem \ref{thm:explicit} proves that this lower bound is tight up to a constant factor if $\Phi$ is unbiased but not transitive-symmetric, but it is unclear whether the bound is achieved exactly by any voting rules that are unbiased but not transitive-symmetric.
	
	\subsection{Graphic Voting Rules}
	
	We consider voting rules that equate the set of voters with the vertices of an underlying graph $G$. To avoid treating ties in a voting model that only admits strict preferences, we assume $\abs{V} = n$ is odd. The {\em graphic voting rule} associated to a graph $G$ of radius $r$ is just $\Phi_G = \Phi_{B_r(G)}$, i.e.\
	\begin{align*}
	\Phi_G(x) = \Phi_{B_r(G)}(x) =
	\begin{cases}
	y & \exists v \in V \,\, \mathrm{s.t.} \,\, x_i = y \,\, \forall i \in B_r(v) \\
	\mathrm{maj}_n(x) & \mathrm{otherwise}
	\end{cases}
	\end{align*}
	
	We make a few remarks about this definition, some of which will motivate us to impose additional assumptions on the graph $G$ underlying the voting rule $\Phi_G$.
	\begin{enumerate}
		\item $\Phi_G$ is a valid voting rule that is odd and monotone because $B_r(G)$ is an intersecting family of sets by our definition of the graph radius.
		
		\item Because $\Phi_G$ is uniquely determined by $G$'s balls of radius $r$, we know that $\mathrm{aut}(\Phi_G) \subseteq \mathrm{aut}(B_r(G))$. So if $B_r(G)$ is not vertex-transitive, then $\Phi_G$ is not transitive-symmetric.
		
		\item We can assume that $\mathrm{diam}(G) = 2$ (and therefore $\mathrm{rad}(G) = 1$) without loss of generality: if $\mathrm{rad}(G) = r$, let $G' = \qty(V, \{ \{u, v\} : u \in V, v \in B_r(u) \})$ be the graph obtained by connecting all the vertices of $G$ in balls of radius $r$. It is easy to see that $\mathrm{rad}(G') = 1$ and $\Phi_G = \Phi_{G'}$. We ignore the trivial case where where $\mathrm{diam}(G') = 1$ and $G'$ is just the complete graph.
		
		\item We further assume that $G$ is $d$-regular and has all balls of radius 1 distinct. Then, every voter in $\Phi_G$ participates in exactly $d + 1$ minimum winning coalitions, each of size $d + 1$.
		
		\item To ensure that $\Phi_G \ne \mathrm{maj}_n$, we further assume that $G$ is $d$-regular with $d \le (n - 3) / 2$.
		
	\end{enumerate}
	
	Theorem \ref{thm:graph_exist} shows that 3 - 5 are reasonable assumptions, even in conjunction. We now provide an explicit example of an odd, monotone voting rule on $n = 11$ voters that is unbiased but not transitive-symmetric:
	\begin{theorem}
		\label{thm:example}
		Let $G$ be the graph depicted in Figure \ref{fig:reg_asymmetric}. Then $\Phi_G$ is an odd, monotone, unbiased voting rule with $\abs{\mathrm{aut}(\Phi_G)} = 1$, i.e.\ $\Phi_G$ is totally asymmetric.
	\end{theorem}
	
	\begin{proof}
		$G$ is a 4-regular graph on 11 vertices with $\mathrm{diam}(G) = 2$ and all balls of radius 1 distinct. One can also check through exhaustive computation that $\mathrm{aut}(B_1(G))$ is trivial. Therefore, $\Phi_G : \{-1, 1\}^{11} \to \{-1, 1\}$ is such that $\mathrm{aut}(\Phi_G) \subseteq \mathrm{aut}(B_1(G)) = \{\mathrm{id}\}$, and each of the 11 voters participates in exactly 5 minimum winning coalitions of size 5. Because these are the only winning coalitions in which $\Phi_G$ differs from majority rule, every voter is pivotal to an equal number of winning coalitions of each size. By Lemma \ref{lem:unbiased}, this implies that $\Phi_G$ is unbiased.
		
	\end{proof}

	\section{A Probabilistic Construction}
	\label{sec:prob}
	Here, we use the probabilistic method to prove the existence of a class of regular graphs whose associated graphic voting rules are totally asymmetric, unbiased, and have winning coalitions of size $\lfloor 0.293n \rfloor$. This result immediately implies our main Theorem \ref{thm:main}. More specifically, let $\mathcal{G}_{n, d}$ be the uniform distribution over $d$-regular graphs on $n$ vertices, where $nd$ is even. For our purposes, this means that $n$ is odd and $d$ is even. Theorem \ref{thm:graph_exist} shows that probability of sampling a graph of interest from $\mathcal{G}_{n, d}$ is $1 - o(1)$, as long as the degree $d$ is large enough:
	
	\begin{theorem}
		\label{thm:graph_exist}
		Let $G \sim \mathcal{G}_{n, d}$ with $\sqrt{n} \log n \le d \le n / 2$. Then, with probability $1 - o(1)$, $\abs{\mathrm{aut}(B_1(G))} = 1$ and $\abs{B_1(u) \cap B_1(v)} = (1 + o(1)) d^2 / n$ for every pair of vertices $u \ne v$. 
	\end{theorem}
	
	When taken together with the following result, our main Theorem \ref{thm:main} follows:
	
	\begin{theorem}
		\label{thm:voting_exist}
		Let $G$ be a $d$-regular graph with $\lfloor 0.293n \rfloor - 1 \le d \le (n - 3) / 2$, $\abs{\mathrm{aut}(B_1(G))} = 1$, and $\abs{B_1(u) \cap B_1(v)} = (1 + o(1)) d^2 / n$ for every pair of vertices $u \ne v$. If $n$ is odd, then $\Phi_G$ is an odd, monotone, unbiased voting rule with $\abs{\mathrm{aut}(\Phi_G)} = 1$. Moreover, $\Phi_G$ has winning coalitions of size $d + 1$.
		
	\end{theorem}
	
	\begin{proof}
		Because $0 < \abs{B_1(u) \cap B_1(v)} = (1 + o(1)) d^2 / n < d + 1$ for all $u \ne v \in V$, we know that $\mathrm{rad}(G) = 1$ and all balls of radius 1 are distinct. Therefore, $\mathrm{aut}(\Phi_G) \subseteq \mathrm{aut}(B_1(G)) = \{\mathrm{id}\}$, and each voter participates in exactly $d + 1$ minimum winning coalitions of size $d + 1$. Now, we prove that $\Phi_G$ is unbiased.
		
		Let $k = \max_{u \ne v} \abs{B_1(u) \cap B_1(v)} = (1 + o(1)) d^2 / n$ be the maximal intersection between any two minimum winning coalitions of $\Phi_G$, and consider the following families of winning coalitions for each $v \in V$:
		\[ \mathcal{F}_v = \{ S \supseteq B_1(v) : \abs{S} < \min \{ 2(d+1) - k, (n+1)/2 \} \} \]
		Because $\abs{S} < 2(d + 1) - k \le \abs{B_1(u) \cup B_1(v)}$ for all $S \in \mathcal{F}$ and $u \ne v \in V$, we know that $\mathcal{F}_u \cap \mathcal{F}_v = \varnothing$ and $\abs{\{S \in \mathcal{F}_u : \abs{S} = t\}} = \abs{\{S \in \mathcal{F}_v : \abs{S} = t\}}$ for all $t$. Moreover, if $S \in \mathcal{F}_v$, then $S \setminus \{u\}$ is not a winning coalition for any $u \in B_1(v)$ (since no $S \in \mathcal{F}_v$ is a majority). In words, every voter in $B_1(v)$ is pivotal to every winning coalition in $\mathcal{F}_v$. 
		
		Now assume that $2 (d + 1) - k \ge (n + 1) / 2$. Then, $\bigcup_{v \in V} \mathcal{F}_v$ is exactly the set of all winning coalitions smaller than a majority. Because every voter $v \in V$ participates in exactly $d + 1$ minimum winning coalitions $B_1(u)$ of size $d + 1$, the analysis above implies that every voter is pivotal for an equal number of winning coalitions of each size. By Lemma \ref{lem:unbiased}, this implies that $\Phi_G$ is unbiased.
		
		Finally, we translate the assumption that $2 (d + 1) - k \ge (n + 1) / 2$ into an assumption about $d$. Since, $k = (1 + o(1)) d^2 / n$, we solve the quadratic inequality $2 (d + 1) - (1 + o(1)) d^2 / n \ge (n + 1) / 2$ to obtain the condition that $d \ge (1 - 1/\sqrt{2} + o(1)) n$. For $n$ large enough, $d \ge \lfloor 0.293n \rfloor - 1$ will suffice.
		
	\end{proof}
	
	The remainder of this section proves Theorem \ref{thm:graph_exist} as a concentration inequality about the {\em defect} of a random regular graph. Intuitively, the defect of a graph $G$ quantifies just how asymmetric its balls $B_1(G)$ of radius 1 are. More formally, we have the following definition and lemma, inspired by \cite{kim2002asymmetry}:
	\begin{definition}
		\label{def:defect}
		Let $G = (V, E)$ be a graph, and let $\sigma, \pi \in S_n$ be permutations of $V$. The {\em defect} of a vertex $v$ with respect to $\sigma$ and $\pi$ is
		\[ D_{\sigma, \pi}(v) = \abs{\sigma(B_1(v)) \,\,\Delta\,\, B_1(\pi(v))} \]
		where $\Delta$ denotes the symmetric difference. Analogously, we define the {\em defect} of the graph $G$ with respect to $\sigma$ and $\pi$ as $D_{\sigma, \pi}(G) = \max_{v \in V} D_{\sigma, \pi}(v)$. 
		
	\end{definition}
	
	\begin{lemma}
		\label{lem:defect}
		Let $G = (V, E)$ be a graph and $\sigma \in S_n$ a permutation of $V$. Then, $\sigma \in \mathrm{aut}(B_1(G))$ if and only if there exists another permutation $\pi \in S_n$ such that $D_{\sigma, \pi}(G) = 0$. Moreover, $\sigma \in \mathrm{aut}(G)$ if and only if $D_{\sigma, \sigma}(G) = 0$. 
		
	\end{lemma}
	\begin{proof}
		By definition, $\sigma$ is an automorphism of $B_1(G)$ if and only if for every vertex $v \in V$, there is another vertex $v' \in V$ such that $\sigma(B_1(v)) = B_1(v')$. In the case where all balls $B_1(v)$ are distinct, this is possible if and only if the mapping $v \mapsto v'$ is a permutation, which we will call $\pi$. In the general case, $\pi$ need not be bijective, but there will always be a bijective $\pi$ for which the desired statement holds biconditionally. Since $\sigma(B_1(v)) = B_1(\pi(v))$ for all $v \in V$, $D_{\sigma, \pi}(G) = 0$. The second statement follows from the fact that $\sigma \in \mathrm{aut}(G)$ if and only if $\{\sigma(u), \sigma(v)\} \in E$ for every edge $\{u, v\} \in E$. This is equivalent to saying that $B_1(\sigma(v)) = \sigma(B_1(v))$ for every vertex $v \in V$.
		
	\end{proof}
	
	Our proof of Theorem \ref{thm:graph_exist} closely follows the approach of \cite{kim2002asymmetry}. Subsection \ref{sec:gnp} proves a concentration inequality about the defect of an Erd\H{o}s-R{\'e}nyi random graph, and Subsection \ref{sec:rand_reg} uses that result to prove an analogous statement (from which Theorem \ref{thm:graph_exist} follows) about random regular graphs.
	
	\subsection{A Result About Erd\H{o}s-R{\'e}nyi Random Graphs}
	\label{sec:gnp}
	The Erd\H{o}s-R{\'e}nyi random graph distribution $G(n, p)$ is the probability space over graphs with vertex set $V = [n]$, where each of the $\binom{n}{2}$ possible edges is present independently with probability $p$. In this subsection, we prove the following concentration inequality about the defect of a graph drawn from the Erd\H{o}s-R{\'e}nyi distribution $G(n, p)$. The proof closely follows that of Theorem 3.1 in \cite{kim2002asymmetry}:
	
	\begin{theorem}
		\label{thm:gnp_asymmetry}
		Let $G \sim G(n, p)$ with $p = \omega(\log n / n)$ and $1 - p = \omega(\log n / n)$. Then, there exist permutations $\sigma, \pi \in S_n$ such that $\abs{ \{v \in V : \sigma(v) \ne v \mbox{ or } \pi(v) \ne v \} } \ge \ell$ and $D_{\sigma, \pi}(G) < (2 - o(1)) n p (1 - p)$ with probability at most
		\begin{align*}
		4 \sum_{k = \ell}^{n} \exp(-c \epsilon^2 n p (1 - p) + 2 \log n)^k = o(1)
		\end{align*}
		for some constant $c > 0$ and any $\epsilon = \epsilon(n)$ such that $\epsilon = o(1)$ and $\epsilon^2 n p (1 - p) = \omega(\log n)$. For example, $p = \log(n)^2 / n$ and $\epsilon = \log(n)^{-1/3}$ suffice.
		
	\end{theorem}
	
	Intuitively, Theorem \ref{thm:gnp_asymmetry} implies that $B_1(G)$ from almost every $G \sim G(n, p)$ is highly asymmetric. We state the probability bound in a particularly technical way to facilitate the proof of an analogous theorem about random $np$-regular graphs in Subsection \ref{sec:rand_reg}. The centerpiece of our proof of Theorem \ref{thm:gnp_asymmetry} is the following concentration inequality, which represents a refinement of McDiarmid's bounded difference inequality \cite{mcdiarmid1989bdi} to the special case of Bernoulli random variables:
	
	\begin{lemma}[Alon et al.\ \cite{alon1997nearly}]
		\label{lem:large_dev}
		Let $X$ be a random variable on a probability space generated by finitely many independent Bernoulli random variables $Y_i \sim \mathrm{Bern}(p_i)$. Let $\delta$ be such that changing any $Y_i$ (keeping all others the same) can change $X$ by at most $\delta$. Define $\Delta^2 = \delta^2 \sum_{i} p_i (1 - p_i)$. Then, for all $0 < t < 2 \Delta / \delta$, $\P{\abs{X - \E[]{X}} > t \Delta} \le 2 e^{-t^2/4}$. 
		
	\end{lemma}
	
	\begin{proof}[Proof (Theorem \ref{thm:gnp_asymmetry}).]
		Let $\sigma, \pi \in S_n$ be permutations of the vertices of $G$, and let $U = \{ v \in V : \sigma(v) \ne v \mbox{ or } \pi(v) \ne v \}$ be the vertices not fixed in place by $\sigma$ and $\pi$. Let $k = \abs{U}$, and assume that $k > 0$ for non-triviality. Next, let $X = \sum_{v \in U} D_{\sigma, \pi}(v)$. We will apply Lemma \ref{lem:large_dev} to the random variable $X$ later in the proof.
		
		For any vertex $v \in U$, let $E_1^{uv}$ be the event that $u \in \sigma(B_1(v)) \setminus B_1(\pi(v))$ and let $E_2^{uv}$ be the event that $u \in B_1(\pi(v)) \setminus \sigma(B_1(v))$. $E_1^{uv}$ occurs if and only if $\{\sigma^{-1}(u), v\} \in E$ and $\{u, \pi(v)\} \notin E$, while $E_2^{uv}$ occurs if and only if $\{u, \pi(v)\} \in E$ and $\{\sigma^{-1}(u), v\} \notin E$. Since $v$ is moved by at least one of $\sigma$ or $\pi$ (by the assumption that $v \in U$), this implies that
		\begin{align*}
		\E[]{D_{\sigma, \pi}(v)} &= \sum_{u \in V} \P{E_1^{uv} \cup E_2^{uv}}
		= \begin{cases}
		2 (n - 1) p (1 - p) & \pi(\sigma(v)) = v \\
		2 (n - 2) p (1 - p) & \mbox{otherwise}
		\end{cases}
		\end{align*}
		where the terms $(n - 1)$ and $(n - 2)$ account for cases where $u = \pi(v)$ and $\sigma^{-1}(u) = v$. Therefore,
		\begin{align*}
		\E[]{X} = \sum_{v \in U} \E[]{D_{\sigma, \pi}(v)} = (2 - o(1)) k n p (1 - p)
		\end{align*}
		Now, we analyze how the presence of edges in $G$ influences the value of $X$. First, we note that $X$ only depends on edges of the graph incident to vertices in $U$. There are $k (n - k) + \binom{k}{2}$ such edges. Second, adding or deleting an edge $\{u, v\}$ from $G$ can only change the values of the following six summands, each by at most 1: $D_{\sigma, \pi}(u)$, $D_{\sigma, \pi}(v)$, $D_{\sigma, \pi}(\sigma^{-1}(u))$, $D_{\sigma, \pi}(\sigma^{-1}(v))$, $D_{\sigma, \pi}(\pi^{-1}(u))$, and $D_{\sigma, \pi}(\pi^{-1}(v))$. This means that $X$ satisfies the hypothesis of Lemma \ref{lem:large_dev} with parameters $\delta = 6$ and $\Delta^2 = 36 (k (n - k) + \binom{k}{2}) p (1 - p) = \Theta(k n p(1 - p))$. So for some constant $c > 0$,
		\begin{align*}
		\P{\abs{X - \E[]{X}} > 2 \epsilon k n p (1 - p)} \le \exp(-c \epsilon^2 k n p (1 - p))	
		\end{align*}
		Therefore, with probability at least $1 - \exp(-c \epsilon^2 k n p (1 - p))$, 
		\begin{align*}
		D_{\sigma, \pi}(G) \ge \max_{v \in U} D_{\sigma, \pi}(v) \ge \frac{1}{k} \qty(\E[]{X} - 2 \epsilon k n p (1 - p)) = (2 - o(1)) n p (1 - p)
		\end{align*}
		To bound the probability that this bound holds for all $\sigma, \pi \in S_n$ for which $k = \abs{U} \ge \ell$, we will take a union bound. There are at most $\binom{n}{k} k! \le n^k$ permutations that exactly fix $n - k$ vertices, so there are at most $\sum_{i = 0}^{k} \sum_{j = 0}^{k} n^i n^j \le 4 n^{2k}$ {\em pairs} of permutations $\pi$ and $\sigma$ that fix exactly $n - k$ vertices between them, i.e.\ for which $\abs{U} = k$. Therefore, there exists a non-trivial pair of permutations $\pi$ and $\sigma$ for which $\abs{U} \ge \ell$ and $D_{\sigma, \pi}(G) < (2 - o(1)) n p (1 - p)$ with probability at most
		\begin{align*}
		\sum_{k = \ell}^{n} 4 n^{2k} \exp(-c \epsilon^2 k n p (1 - p)) = 4 \sum_{k = \ell}^{n} \exp(-c \epsilon^2 n p (1 - p) + 2 \log n)^k
		\end{align*}
		
	\end{proof}
	
	\subsection{Proof of Theorem \ref{thm:graph_exist}}
	\label{sec:rand_reg}
	
	Theorem \ref{thm:graph_exist} follows immediately from Lemmas \ref{lem:codegree} and \ref{lem:rand_reg_asymmetry}:
	
	\begin{lemma}[Krivelevich et al.\ \cite{krivelevich2001random}]
		\label{lem:codegree}
		Let $G \sim \mathcal{G}_{n, d}$. If $\sqrt{n} \log n \le d \le n - n / \log_2 n$, then with probability $1 - o(1)$, $\abs{B_1(u) \cap B_1(v)} = (1 + o(1)) d^2 / n$ for every pair of vertices $u \ne v$ in $G$. If instead $\log n \le d \le \sqrt{n} \log n$, then with probability $1 - o(1)$, there exists a constant $\epsilon > 0$ such that $\abs{B_1(u) \cap B_1(v)} < d^{1-\epsilon}$ for every pair of vertices $u \ne v$ in $G$.
	\end{lemma}
	
	\begin{lemma}
		\label{lem:rand_reg_asymmetry}
		Let $G \sim \mathcal{G}_{n, d}$, with $d = \omega(\log n)$ and $d \le n / 2$. With probability $1 - o(1)$, $\abs{\mathrm{aut}(B_1(G))} = 1$, i.e.\ $B_1(G)$ is totally asymmetric.
	\end{lemma}
	
	To prove Lemma \ref{lem:rand_reg_asymmetry}, we will need two additional technical lemmas from \cite{kim2002asymmetry}:
	
	\begin{lemma}[Kim et al.\ \cite{kim2002asymmetry}]
		\label{lem:set_edge_span}
		Let $G \sim \mathcal{G}_{n, d}$ with $d \le n^{3/4}$ and $d = \omega(\log n)$. Then, with probability $1 - o(1)$, every subset of vertices in $G$ of size $a \le n d^{-1/3}$ spans at most $d a / \log d$ edges.
	\end{lemma}
	
	\begin{lemma}[Kim et al.\ \cite{kim2002asymmetry}]
		\label{lem:gnp_regular}
		Let $G \sim G(n, d/n)$. For any constant $\delta > 0$ and $n$ sufficiently large, $G$ is $d$-regular with probability at least $\exp(-nd^{1/2+\delta})$. 
	\end{lemma}
	
	\begin{proof}[Proof (Lemma \ref{lem:rand_reg_asymmetry}).]
		The proof closely follows Section 5 of \cite{kim2002asymmetry}, which proves an analogous result for $\mathrm{aut}(G)$, rather than $\mathrm{aut}(B_1(G))$. Our strategy for adapting their proof is similar to the way we proved Theorem \ref{thm:gnp_asymmetry}. Re-introducing notation from that proof, for any non-trivial pair permutations $\sigma, \pi \in S_n$, we will let $U = \{ v \in V : \pi(v) \ne v \mbox{ or } \sigma(v) \ne v \}$ and $W = V \setminus U$. The following observation is critical:
		\begin{align}
		D_{\sigma, \pi}(v) &\ge \abs{(\sigma(B_1(v)) \cap W) \,\,\Delta\,\, (B_1(\pi(v)) \cap W) } \notag \\
		&= \qty(d + 1 - \abs{\sigma(B_1(v)) \cap U}) + \qty(d + 1 - \abs{B_1(\pi(v)) \cap U}) \notag \\
		& \quad -2 \abs{\sigma(B_1(v)) \cap B_1(\pi(v)) \cap W} \notag \\
		&= 2d + 2 - \abs{B_1(v) \cap U} - \abs{B_1(\pi(v)) \cap U} - 2 \abs{B_1(v) \cap B_1(\pi(v)) \cap W} \label{eqn:nonobvious} \\
		&\ge 2d - \abs{B_1(v) \cap U} - \abs{B_1(\pi(v)) \cap U} - 2 \abs{B_1(v) \cap B_1(\pi(v))} \label{eqn:defect}
		\end{align}
		Inequality \ref{eqn:nonobvious} follows because that $\sigma(v) \in U$ for all $v \in U$ and $\sigma(v) = v$ for all $v \in V \setminus U$. Now, we enumerate cases.
		
		\begin{enumerate}
			\item $n^{3/4} \le d \le n / 2$. Then, Lemma \ref{lem:codegree} tells us that $\abs{B_1(v) \cap B_1(\pi(v))} = (1 + o(1)) d^2 / n$ with probability $1 - o(1)$. If we let $\ell = d^{1-\delta} = o(d)$ for some small constant $\delta > 0$ ($\delta = 1/100$ suffices), we need to analyze two additional sub-cases:
			\begin{enumerate}
				\item $\abs{U} > \ell$. For any graph $G$, let $B_G$ denote the event that there exist permutations $\sigma, \pi \in S_n$ such that $\abs{U} > \ell$ and $D_{\sigma, \pi}(G') < (2 - o(1)) d (1 - d/n)$. By Theorem \ref{thm:gnp_asymmetry},
				\begin{align*}
				\P[G' \sim G(n, d/n)]{B_{G'}} &\le 4 \sum_{k = \ell + 1}^{n} \exp(-c \epsilon^2 d (1 - d/n) + 2 \log n )^k
				\end{align*}
				for an appropriately chosen constant $c$. If we constrain $\epsilon \ge d^{-\delta}$, our assumption that $\ell = d^{1-\delta}$ tells us that this probability is at most $\exp(-\Omega(d^{2-3\delta}))$ and therefore
				\begin{align*}
				\P[G \sim \mathcal{G}_{n, d}]{B_G} &= \P[G' \sim G(n, d/n)]{B_{G'} \mid G' \,\, d\mbox{-regular}} \\
				&\le \P[G' \sim G(n, d/n)]{B_{G'}} / \P[G' \sim G(n, d/n)]{G' \,\, d\mbox{-regular}} \\
				&\le \exp(-\Omega(d^{2 - 3\delta}) + n d^{1/2 + \delta})
				\end{align*}
				The second-to-last line uses the fact that $\P{A \mid B} \le \P{A} / \P{B}$, and the last line invokes Lemma \ref{lem:gnp_regular}, which says that $G'$ is regular with probability at least $\exp(-nd^{1/2+\delta})$. Finally, because we constrain $n^{3/4} \le d \le n / 2$, the final quantity is $o(1)$ as desired.
				
				\item $\abs{U} \le \ell$. Then, $\abs{B_1(v) \cap U} \le \abs{U} = o(d)$ and $\abs{B_1(\pi(v)) \cap U} \le \abs{U} = o(d)$. Substituting these quantities into Inequality \ref{eqn:defect}, we get that with probability $1 - o(1)$, $D_{\sigma, \pi}(v) \ge (2 - o(1)) d (1 - d/n)$ for all non-trivial pairs of permutations $\sigma$ and $\pi$ for which $\abs{U} \le \ell$.
				
			\end{enumerate}
			
			\item $d \le n^{3/4}$ and $d = \omega(\log n)$. Then, by Lemma \ref{lem:codegree}, $\abs{B_1(v) \cap B_1(\pi(v))} = o(d)$ for all vertices $v$ with probability $1 - o(1)$. If we let $\ell = n d^{-1/3} \ge d$, we have two-sub-cases again:
			
			\begin{enumerate}
				\item $\abs{U} > \ell$. The analysis is nearly identical to case 1(a), and proves the same result.
				
				\item $\abs{U} \le \ell$. By Lemma \ref{lem:set_edge_span}, with probability $1 - o(1)$, {\em every} subset $U \subseteq V$ of vertices with $\abs{U} \le \ell$ spans at most $d \abs{U} / \log d = o(d)$ edges. This implies that with probability $1 - o(1)$, {\em every} such $U$ contains a vertex $v \in U$ such that $\abs{B_1(v) \cap U} = o(d)$ and $\abs{B_1(\pi(v)) \cap U} = o(d)$. So substituting into Inequality \ref{eqn:defect}, we get that with probability $1 - o(1)$, $D_{\sigma, \pi}(G) \ge (2 - o(1)) d \ge (2 - o(1)) d (1 - d/n)$ for every non-trivial pair of permutations $\sigma$ and $\pi$ for which $\abs{U} \le \ell$. 
				
			\end{enumerate}
			
		\end{enumerate}
		
		Therefore, we can conclude that with probability $1 - o(1)$, $D_{\sigma, \pi}(G) \ge (2 - o(1)) d (1 - d/n)$ for every non-trivial pair permutations $\sigma$ and $\pi$. By Lemma \ref{lem:defect}, this result implies that $\abs{\mathrm{aut}(B_1(G))} = 1$ with probability $1 - o(1)$. 
		
	\end{proof}
	
	\section{An Explicit Construction}
	\label{sec:explicit}
	
	In this section, we explicitly construct a voting rule that is unbiased, neither transitive-symmetric nor asymmetric, and has winning coalitions of size $O(\sqrt{n})$:
	
	\begin{theorem}
		\label{thm:explicit}
		Let $m = q^2 + q + 1$ for some prime power $q$, and let $n = 11m$. Then, there exists an odd, monotone, unbiased voting rule $\Phi : \{-1, 1\}^n \to \{-1, 1\}$ that is not transitive symmetric and has winning coalitions of size $\lfloor 1.508 \sqrt{n} \rfloor + 5$.
		
	\end{theorem}
	
	We prove Theorem \ref{thm:explicit} by using the procedure we describe in Lemma \ref{lem:compose} to compose the unbiased but asymmetric graphic voting rule $\Phi_1$ of Theorem \ref{thm:example} with a transitive-symmetric voting rule $\Phi_2$ described by \cite{tamuz2019equitable}. Informally, we consider $n_2$ groups of $n_1$ voters each. We let each of these groups make a decision using $\Phi_1$ and then aggregate these group decisions using $\Phi_2$. More formally, 
	
	\begin{lemma}
		\label{lem:compose}
		Let $\Phi_1 : \{-1, 1\}^{n_1} \to \{-1, 1\}$ be an odd, monotone, unbiased voting rule that is not transitive-symmetric and has winning coalitions of size $d_1$. Let $\Phi_2 : \{-1, 1\}^{n_2} \to \{-1, 1\}$ be an odd, monotone, transitive-symmetric voting rule that has winning coalitions of size $d_2$. Then, there exists a voting rule $\Phi : \{-1, 1\}^{n_1 \times n_2} \to \{-1, 1\}$ that has winning coalitions of size $d_1 d_2$, and is odd, monotone, unbiased, and not transitive-symmetric.
	\end{lemma}
	
	\begin{proof}
		For $X \in \{-1, 1\}^{n_1 \times n_2}$, let $X^{(i)} = (x_{1,i}, \ldots, x_{n_1,i})^T \in \{-1, 1\}^{n_2}$ be the $i^{th}$ column vector of $X$, and define
		\[ \Phi(X) = \Phi_2\qty(\Phi_1(X^{(1)}), \ldots, \Phi_1(X^{(n_2)})) \]
		$\Phi$ is odd and monotone because $\Phi_1$ and $\Phi_2$ are also odd and monotone. We can form a winning coalition of size $d_1 d_2$ for $\Phi$ by taking a winning coalition for $\Phi_2$ and assigning to each member of that coalition an entire winning coalition for $\Phi_1$.
		
		Now, we show that $\Phi$ is unbiased. Because $\Phi_1$ is unbiased, every $x_{i,j}$ has an equal probability of being pivotal to $\Phi_1(X^{(j)})$. Because $\Phi_2$ is unbiased, every $\Phi_1(X^{(j)})$ has an equal probability of being pivotal to $\Phi(X) = \Phi_2(\Phi_1(X^{(1)}), \ldots, \Phi_1(X^{(n_2)}))$. So every $x_{i,j}$ has an equal probability of being pivotal to $\Phi(X)$, i.e.\ $\Phi$ is unbiased.
		
		Finally, we show that $\Phi$ is not transitive-symmetric. Let $\sigma \in \mathrm{aut}(\Phi)$, so $\Phi(X) = \Phi(\sigma^{-1}(X))$ for all $X \in \{-1, 1\}^{n_1 \times n_2}$. By the definition of $\Phi$ and the fact that $\Phi_2$ is non-constant, each $X$ must have some $\sigma_{2,X} \in \mathrm{aut}(\Phi_2)$ such that
		\begin{align*}
		\qty(\Phi_1(\sigma^{-1}(X)^{(1)}), \ldots, \Phi_1(\sigma^{-1}(X)^{(n_2)})) = \qty(\Phi_1(X^{\sigma_{2,X}^{-1}(1)}), \ldots, \Phi_1(X^{\sigma_{2,X}^{-1}(n_2)}))
		\end{align*}
		Because this must hold for all $X$ and $\Phi_1$ is also non-constant, this further implies that each $X$ must also have some $\sigma_{1,X} \in \mathrm{aut}(\Phi_1)$ such that 
		\[ \sigma^{-1}(X)^{(i)} = \sigma_{1,X}^{-1}(X^{(\sigma_{2,X}^{-1}(i))})\]
		for all $i \in [n_2]$. However, because $\Phi_1$ is not transitive-symmetric, there must be some $u, v \in [n_1]$ such that $\sigma_{1,X}(u) \ne v$ for all $\sigma_{1,X} \in \mathrm{aut}(\Phi_1)$. Therefore, we can conclude that $\Phi$ is not transitive-symmetric. 
		
	\end{proof}
	
	\begin{proof}[Proof (Theorem \ref{thm:explicit})]
		Theorem \ref{thm:example} describes a voting rule $\Phi_1 : \{-1, 1\}^{11} \to \{-1, 1\}$ that is odd, monotone, unbiased, not transitive-symmetric, and has winning coalitions of size 5. Theorem 6 of \cite{tamuz2019equitable} constructs a voting rule $\Phi_2 : \{-1, 1\}^m \to \{-1, 1\}$ that is odd, monotone, transitive-symmetric, and has winning coalitions of size $\lceil \sqrt{m} \rceil$. Lemma \ref{lem:compose} provides an algorithm to compose $\Phi_1$ and $\Phi_2$ to obtain a voting rule $\Phi : \{-1, 1\}^{11m} \to \{-1, 1\}$ that is odd, monotone, unbiased, not transitive-symmetric, and has winning coalitions of size $5 \lceil \sqrt{m} \rceil \le \frac{5 \sqrt{n}}{\sqrt{11}} + 5 \le \lfloor 1.508 \sqrt{n} \rfloor + 5$.
		
	\end{proof}
	
	\section{Conclusions and Future Work}
	\label{sec:conclusion}
	We have provided two proofs that there exist voting rules which are unbiased but not transitive-symmetric. At the crux of our arguments is a class of voting rules defined by $d$-regular graphs of diameter 2.
	
	Section \ref{sec:prob} provides a probabilistic proof that there exist voting rules that are unbiased and totally asymmetric (not just not transitive-symmetric). Theorem \ref{thm:graph_exist} shows that with probability $1 - o(1)$, the voting rules defined by random $d$-regular graphs for $d \ge \sqrt{n} \log n$ are totally asymmetric and have all voters participate in exactly $d + 1$ minimum winning coalitions of size $d + 1$. Theorem \ref{thm:voting_exist} shows that if we relax the lower bound on the degree to $d \ge 0.293n - 1$, these voting rules are also unbiased with probability $1 - o(1)$. 
	
	Section \ref{sec:explicit} explicitly constructs a class of voting rule that is unbiased, but neither transitive-symmetric nor totally asymmetric. On the other hand, this class of voting rules has much smaller winning coalitions of size $O(\sqrt{n})$, which differs by only a constant factor from the lower bound of $\lceil \sqrt{n} \rceil$ that we proved in Lemma \ref{lem:lowerbound}.
	
	We see two possible directions for future work. First, we conjecture that there exist voting rules that are unbiased, totally asymmetric, and have winning coalitions of size $\sqrt{n} \log n + 1$ (i.e.\ the lower bound of Theorem \ref{thm:voting_exist} can be made to match that of Theorem \ref{thm:graph_exist}). A more sophisticated proof of Lemma \ref{lem:codegree} may be a fruitful avenue to such a result.
	
	Second, given the existence of voting rules that are unbiased, neither transitive-symmetric nor totally asymmetric, and have winning coalitions of size $O(\sqrt{n})$, we conjecture that Theorems \ref{thm:graph_exist} and \ref{thm:voting_exist} can be refined to prove that there exist $d$-regular graphs whose associated voting rules are unbiased but not transitive-symmetric (rather than totally asymmetric) as long as $d \ge \lceil \sqrt{n} \rceil - 1$ (rather than $d \ge \sqrt{n} \log n$). In other words, we believe that the lower bound of Lemma \ref{lem:lowerbound} is tight not only in general, but also for our construction using graphic voting rules. 
	
	\section*{Acknowledgments}
	I would like to thank Omer Tamuz for pointing me to some useful references and for his helpful feedback on drafts of this paper. I would like to thank Leonard Schulman for his assistance in verifying the correctness of some of the proofs in Section \ref{sec:prob}.
	
	\bibliographystyle{amsplain}
	\bibliography{references}

\providecommand{\bysame}{\leavevmode\hbox to3em{\hrulefill}\thinspace}
\providecommand{\MR}{\relax\ifhmode\unskip\space\fi MR }
\providecommand{\MRhref}[2]{%
  \href{http://www.ams.org/mathscinet-getitem?mr=#1}{#2}
}
\providecommand{\href}[2]{#2}
\begin{thebibliography}{1}

\bibitem{alon1997nearly}
Noga Alon, Jeong-Han Kim, and Joel Spencer, \emph{Nearly perfect matchings in
  regular simple hypergraphs}, Israel Journal of Mathematics \textbf{100}
  (1997), 171--187.

\bibitem{tamuz2019equitable}
Laurent Bartholdi, Wade Hann-Caruthers, Maya Josyula, Omer Tamuz, and Leeat
  Yariv, \emph{Equitable voting rules}, Working paper, 2019.

\bibitem{isbell1960homogeneous}
John Isbell, \emph{Homogeneous games {II}}, Proceedings of the American
  Mathematical Society \textbf{11} (1960), 159--161.

\bibitem{kim2002asymmetry}
Jeong~Han Kim, Benny Sudakov, and Van Vu, \emph{On the asymmetry of random
  regular graphs and random graphs}, Random Structures and Algorithms
  \textbf{21} (2002), 216--224.

\bibitem{krivelevich2001random}
Michael Krivelevich, Benny Sudakov, Van Vu, and Nicholas Wormald, \emph{Random
  regular graphs of high degree}, Random Structures and Algorithms \textbf{18}
  (2001), 346--363.

\bibitem{mcdiarmid1989bdi}
Colin McDiarmid, \emph{On the method of bounded differences}, Surveys in
  Combinatorics \textbf{141} (1989), 148--188.

\end{thebibliography}
	
\end{document}